\documentclass[10]{article}
\usepackage{amsfonts,amssymb,amsmath,amsthm,cite}
\usepackage{graphicx}

\usepackage[applemac]{inputenc}
\usepackage{amsmath,amssymb,amsthm, hyperref, euscript}
\usepackage[matrix,arrow,curve]{xy}
\usepackage{graphicx}
\usepackage{tabularx}
\usepackage{float}
\usepackage{hyperref}
\usepackage{tikz}
\usepackage{slashed}
\usepackage{mathrsfs}

\usetikzlibrary{matrix}

\usepackage[T1]{fontenc}
\usepackage{amsfonts,cite}
\usepackage{graphicx}

\usepackage[applemac]{inputenc}

\usepackage[sc]{mathpazo}
\DeclareFontFamily{T1}{pzc}{}
\DeclareFontShape{T1}{pzc}{m}{it}{1.8 <-> pzcmi8t}{}
\DeclareMathAlphabet{\mathpzc}{T1}{pzc}{m}{it}

\theoremstyle{plain}
\newtheorem{prop}{Proposition}[section]

\newtheorem{cor}[prop]{Corollary}
\newtheorem{thm}[prop]{Theorem}
\newtheorem{theorem}[prop]{Theorem}
\newtheorem{lemma}[prop]{Lemma}

\theoremstyle{definition}
\newtheorem{defn}[prop]{Definition}
\newtheorem{rem}[prop]{Remark}

\theoremstyle{definition}
\newtheorem{definition}[prop]{Definition}

\newtheorem{remark}[prop]{Remark}
\numberwithin{equation}{section}

\DeclareMathOperator{\Dom}{\mathrm{Dom}}              

\newcommand{\vertiii}[1]{{\left\vert\kern-0.25ex\left\vert\kern-0.25ex\left\vert #1
    \right\vert\kern-0.25ex\right\vert\kern-0.25ex\right\vert}}
\newcommand{\Ga}{\Gamma}                     
\newcommand{\Coo}{C^\infty}                  

\usepackage[sc]{mathpazo}
\newbox\ncintdbox \newbox\ncinttbox 
\setbox0=\hbox{$-$} \setbox2=\hbox{$\displaystyle\int$}
\setbox\ncintdbox=\hbox{\rlap{\hbox
    to \wd2{\hskip-.125em \box2\relax\hfil}}\box0\kern.1em}
\setbox0=\hbox{$\vcenter{\hrule width 4pt}$}
\setbox2=\hbox{$\textstyle\int$} \setbox\ncinttbox=\hbox{\rlap{\hbox
    to \wd2{\hskip-.175em \box2\relax\hfil}}\box0\kern.1em}

\newcommand{\Id}{\mathrm{Id}}                

\newcommand{\A}{\mathcal{A}}                 
\newcommand{\E}{\mathcal{E}}                 

\newcommand{\C}{\mathbb{C}}                  
\renewcommand{\H}{\mathcal{H}}               
\newcommand{\hookto}{\hookrightarrow}        

\newcommand{\om}{\omega}                     

\newcommand{\T}{\mathbb{T}}                  

\newcommand{\sA}{\mathcal{A}}       
\newcommand{\sS}{\mathcal{S}}       

\newcommand{\Om}{\Omega}       

\newcommand{\Th}{\Theta}          

\def\<#1|#2>{\langle#1\stroke#2\rangle} 
\def\?#1|#2?{\{#1\stroke#2\}}        

\def\<#1,#2>{\langle#1,#2\rangle}            
\def\ee_#1{e_{{\scriptscriptstyle#1}}}       
\def\wick:#1:{\mathopen:#1\mathclose:}       

\newbox\ncintdbox \newbox\ncinttbox 
\setbox0=\hbox{$-$}
\setbox2=\hbox{$\displaystyle\int$}
\setbox\ncintdbox=\hbox{\rlap{\hbox
           to \wd2{\box2\relax\hfil}}\box0\kern.1em}
\setbox0=\hbox{$\vcenter{\hrule width 4pt}$}
\setbox2=\hbox{$\textstyle\int$}
\setbox\ncinttbox=\hbox{\rlap{\hbox
           to \wd2{\hskip-.05em\box2\relax\hfil}}\box0\kern.1em}

\newcommand{\stroke}{\mathbin|}   

\newcommand{\End}{\mathrm{End}}       
\newcommand{\Hom}{\mathrm{Hom}}       

\newcommand{\be}{\begin{equation}}
\renewcommand{\ee}{\end{equation}}
\newcommand{\bea}{\begin{eqnarray}}
\newcommand{\eea}{\end{eqnarray}}
\newcommand{\bean}{\begin{eqnarray*}}
	\newcommand{\eean}{\end{eqnarray*}}
\newcommand{\brray}{\begin{array}}
	\newcommand{\erray}{\end{array}}

\title{Finite Nocommutative Coverings \\and \\ Flat Connections}
\begin{document}
\maketitle  \setlength{\parindent}{0pt}
\begin{center}
\author{
{\textbf{Petr R. Ivankov*}\\
e-mail: * monster.ivankov@gmail.com \\
}
}
\end{center}

\vspace{1 in}

\noindent

\paragraph{}	

Any flat connection on a principal fibre bundle comes from a linear representation of the fundamental group. The noncommutative analog of this fact is discussed here. 

\section{Motivation. Preliminaries}
\subsection{Coverings}
\begin{defn}\label{comm_cov_pr_defn}\cite{spanier:at}
	Let $\widetilde{\pi}: \widetilde{\mathcal{X}} \to \mathcal{X}$ be a continuous map. An open subset $\mathcal{U} \subset \mathcal{X}$ is said to be {\it evenly covered } by $\widetilde{\pi}$ if $\widetilde{\pi}^{-1}(\mathcal U)$ is the disjoint union of open subsets of $\widetilde{\mathcal{X}}$ each of which is mapped homeomorphically onto $\mathcal{U}$ by $\widetilde{\pi}$. A continuous map $\widetilde{\pi}: \widetilde{\mathcal{X}} \to \mathcal{X}$ is called a {\it covering projection} if each point $x \in \mathcal{X}$ has an open neighborhood evenly covered by $\widetilde{\pi}$. $\widetilde{\mathcal{X}}$ is called the {
		\it covering space} and $\mathcal{X}$ the {\it base space} of the covering.
\end{defn}
\begin{defn}\label{cov_proj_cov_grp}\cite{spanier:at}
	Let $p: \mathcal{\widetilde{X}} \to \mathcal{X}$ be a covering.  A self-equivalence is a homeomorphism $f:\mathcal{\widetilde{X}}\to\mathcal{\widetilde{X}}$ such that $p \circ f = p$. This group of such homeomorphisms is said to be the {\it group of covering transformations} of $p$ or the {\it covering group}. Denote by $G\left( \mathcal{\widetilde{X}}~|~\mathcal{X}\right)$ this group.
\end{defn}


\begin{rem}
	Above results are copied from \cite{spanier:at}. Below  the \textit{covering projection} word is replaced with \textit{covering}.
	
\end{rem}

\subsection{Flat connections in the differential geometry}\label{geom_flat_subsec}
\paragraph*{}
Here I follow to \cite{koba_nomi:fgd}. Let $M$ be a manifold and $G$ a Lie group. A (\textit{differentiable}) \textit{principal bundle over M with group} $G$ consists of a manifolfd $P$ and an action of $G$ on $P$ satisfying the following conditions:
\begin{enumerate}
	\item [(a)] $G$ acts freely on $P$ on the right: $\left(u, a \right) \in P \times G \mapsto ua = R_au \in P$;
	\item[(b)] $M$ is the quotient space of $P$ by the equivalence relation induced by $G$, i.e. $M = P/G$, and the canonical projection $\pi: P \to M$ is differentiable;
	\item[(c)] $P$ is locally trivial, that is, every point $x$ of $M$ has an open neighborhood $U$ such that $\pi^{-1}\left( U\right)$ is isomophic to  $U\times G$ in the sense that there is a diffeomorphism $\psi:  \pi^{-1}\left( U\right) \to U \times G$ such that $\psi\left( u\right) = \left( \pi\left( u\right), \varphi\left(u \right) \right) $ where $\varphi$ is a mapping of $\pi^{-1}\left(U \right)$ into $G$ satisfying  $\psi\left(ua \right)= \left( \psi\left( u\right)\right) a$  for all $u \in \pi^{-1}\left(U \right)$ and $a \in G$. 
\end{enumerate}
A principal fibre bundle will be denoted by $P\left( M, G, \pi\right), ~ P\left(M, G \right)$ or simply $P$.
\paragraph*{} Let $P\left(M, G \right)$ be a principal fibre bundle over a manifold with group $G$. For each $u \in P$ let $T_u\left(P \right)$ be a tangent space of $P$ at $u$ and $G_u$ the subspace of $T_u\left( P\right)$ consisting of vectors tangent to the fibre through $u$. A \textit{connection} $\Ga$   in $P$ is an assignment of a subspace $Q_u$ of $T_u\left(P \right)$ to each $u \in P$ such that
\begin{enumerate}
	\item [(a)] $T_u\left(P \right) = G_u \oplus Q_u$ (direct sum);
	\item[(b)] $Q_{ua}= \left(R_a \right)_*Q_u$ for every $u \in P$ and $a \in G$, where $R_a$ is a transformation of $P$ induced by $a \in G, ~ R_au=ua$.
\end{enumerate}
\paragraph*{}
Let $P = M \times G$ be a trivial principal bundle. For each $a \in G$, the set $M \times \{a\}$ is a submanifold of $P$. The \text{canonical flat connection} in $P$ is defined by taking the tangent space to $M \times \{a\}$ at $u = \left(x, a \right)$ as the horizontal tangent subspace at $u$. A connection in any principal bundle is called \textit{flat} if every point has a neighborhood such that the induced connection in $P|_U = \pi^{-1}\left(U \right)$ is isomorphic with the canonical flat connection.
\begin{cor}\label{dg_flat_cor}(Corollary II 9.2 \cite{koba_nomi:fgd})
	Let $\Ga$ be a connection in $P\left(M, G \right)$ such that the curvature vanishes identically. If $M$ is paracompact and simply connected, then $P$ is isomorphic to the trivial bundle and $\Ga$ is isomorphic to the canonical flat connection in $M \times G$.
\end{cor}
\paragraph*{}

 If $\widetilde{\pi}: \widetilde{M} \to M$ is a covering then the $\widetilde{\pi}$-\textit{lift} of $P$ is a principal $\widetilde{P}\left(\widetilde{M}, G \right)$  bundle, given by
\be\nonumber
\widetilde{P} = \left\{\left(u, \widetilde{x}\right) \in P \times \widetilde{M}~|~ \pi\left(u \right) = \widetilde{\pi}\left( \widetilde{x}\right) \right\}.
\ee
 If $\Ga$ is a  connection on $P\left( M, G\right)$ and $\widetilde{M} \to M$ is a covering then is a canonical connection $\widetilde{\Ga}$ on $\widetilde{P}\left(\widetilde{M}, G \right)$ which is the \textit{lift} of $\Ga$, that is, for any $\widetilde{u} \in \widetilde{P}$ the horizontal space $\widetilde{Q}_{\widetilde{u}}$ is isomorphically  mapped onto the horizontal space $Q_{\widetilde{\pi}\left(\widetilde{u} \right) }$ associated with the connection $\Ga$.
If $\Ga$ is flat then from the Proposition (II 9.3 \cite{koba_nomi:fgd}) it turns out that there is a covering $\widetilde{M} \to M$ such that $\widetilde{P}\left(\widetilde{M}, G \right)$ (which is the lift of $P\left(M,G\right)$) is a trivial bundle, so the lift $\widetilde{\Ga}$ of $\Ga$ is a canonical flat connection (cf. Corollary \ref{dg_flat_cor}). From the the Proposition (II 9.3 \cite{koba_nomi:fgd}) it follows that for any flat connection $\Ga$  on $P\left(M, G \right)$ there is a group homomorphism $\varphi: G\left( \widetilde{M} ~|~ M\right) \to G$ such that
\begin{enumerate}
	\item [(a)] There is an action $G\left( \widetilde{M} ~|~ M\right) \times \widetilde{P} \to \widetilde{P} \approx \widetilde{M} \times G$ given by
	$$
	g \left(\widetilde{x}, a \right) = \left( g\widetilde{x}, \varphi\left( g\right) a\right); \forall \widetilde{x} \in \widetilde{M}, ~ a \in G, 
	$$
	\item[(b)] There is the canonical diffeomorphism  $P = \widetilde{P}/G\left( \widetilde{M} ~|~ M\right)$,
	\item[(c)] The lift $\tilde{\Ga}$ of $\Ga$ is a canonical flat connection.
\end{enumerate}

\begin{defn}
In the above situation we say that the flat connection $\Ga$ is \textit{induced} by the covering $\widetilde{M}\to M$ and the homomorphism $G\left( \widetilde{M}~|~M\right) \to G$, or we say that $\Ga$ \textit{comes from} $G\left( \widetilde{M}~|~M\right) \to G$.
\end{defn}
\begin{remark}
	The  Proposition (II 9.3 \cite{koba_nomi:fgd}) assumes that $\widetilde{M} \to M$ is the universal covering however it is not always necessary requirement.
\end{remark}
\begin{remark}
	If $\pi_1\left(M, x_0 \right)$ is the fundamental group \cite{spanier:at} then there is the canonical surjective homomorphism $\pi_1\left(M, x_0 \right) \to G\left( \widetilde{M}~|~M\right)$. So there exist the composition $\pi_1\left(M, x_0 \right) \to G\left( \widetilde{M}~|~M\right) \to G$. It follows that any flat connection comes from the homomorphisms $\pi_1\left(M, x_0 \right) \to G$. 
\end{remark} 
 \paragraph*{}
 Suppose that there is the right action of $G$ on $P$ and suppose that $F$ is a manifold with the left action of $G$. There is an action of $G$ on $P \times F$ given by $a\left( u, \xi\right) = \left(u a, a^{-1}\xi \right)$ for any $a \in G$ and $\left( u, \xi\right) \in P\times F$. The quotient space $P \times_G F = \left(P \times F \right)/G$ has the natural structure of a manifold and if $E =  P \times_G F$ then $E\left(M, F, G, P \right)$ is said to be the \textit{fibre bundle over the base $M$, with (standard) fibre $F$, and (structure) group G which is associated with the principal bundle P} (cf. \cite{koba_nomi:fgd}). If $P = M \times G$ is the trivial bundle then $E$ is also trivial, that is, $E = M \times F$. If $F = \C^n$ is a vector space and the action of $G$ on $\C^n$ is a linear representation of the group then $E$ is the linear bundle. Denote by $T\left(M \right)$ (resp. $T^*\left(M \right)$) the tangent  (resp. contangent) bundle, and denote by $\Ga\left( E\right)$, $\Ga\left(T\left(M \right)\right)$, $\Ga\left(T^*\left(M \right)\right)$ the spaces of sections of $E$, $T\left(M \right)$, $T^*\left(M \right)$ respectively. Any connection $\Ga$ on $P$ gives a covariant derivative  on $E$, that is,
for any section  $X \in \Ga\left( T\left(M \right)\right) $ and any section $\xi \in \Ga\left( E\right)$ there is the derivative given by
\be\nonumber
 \nabla_X\left( \xi\right) \in \Ga\left( E\right).
\ee
 If $E = M \times \C^n$, $\Ga$ is the canonical flat connection and $\xi$ is a trivial section, that is, $\xi = M \times \{x\}$ then 
\be\label{comm_triv_eqn}
\nabla_X \xi = 0,~~ \forall X \in T\left(M \right).
\ee
For any connection there is the unique map
\be\label{comm_alg_conn}
\nabla : \Ga\left( E\right) \to \Ga\left( E \otimes T^*\left( M\right) \right)
\ee
such that
\be\nonumber
\nabla_X \xi = \left(\nabla \xi, X \right)
\ee
where the pairing $\left(\cdot, \cdot\right) : \Ga\left( E \otimes T^*\left( M\right) \right) \times \Ga\left( T\left( M\right)\right)  \to \Ga\left( E \right) $ is induced by the pairing $\Ga\left( T^*\left(M \right)\right)   \times \Ga\left( T\left( M\right)\right)  \to \Coo\left(M \right) $.

\subsection{Noncommutative generalization of connections}
\paragraph*{} The noncommutative analog of manifold is a spectral triple and there is the noncommutative analog of connections.

\subsubsection{Connection and curvature}
\begin{defn}\cite{connes:ncg94}
	\begin{enumerate}
		\item [(a)]  A \textit{cycle} of dimension $n$ is a triple $\left(\Om, d, \int \right)$ where $\Om = \bigoplus_{j=0}^n\Om^j$  is a graded algebra over $\C$, $d$ is a graded derivation of degree 1 such that $d^2=0$, and $\int :\Om^n \to \C$ is a closed graded trace on $\Om$,
		\item[(b)] Let $\A$  be an algebra over $\C$. Then a \textit{cycle over} $\A$ is given by a cycle $\left(\Om, d, \int \right)$	and a homomorphism $\A \to \Om^0$.
	\end{enumerate}
\end{defn}
\begin{defn}\label{conn_defn}\cite{connes:ncg94}
	Let $\A\xrightarrow{\rho} \Om$ be a cycle over $\A$, and $\E$ a finite projective module over $\A$.
	Then a \textit{connection} $\nabla$ on $\E$ is a linear map  $\nabla: \E \to \E \otimes_{\A} \Om^1$ such that
	\be\label{conn_prop_eqn}
	\nabla\left(\xi x \right) =  \nabla\left(\xi \right) x =  \xi \otimes d\rho\left(x \right) ; ~ \forall \xi \in \E, ~ \forall x \in \A.
	\ee
\end{defn}
Here $\E$ is a right module over $\A$ and $\Om^1$ is considered as a bimodule over $\A$.
\begin{rem}
	The map $\nabla: \E \to \E \otimes_{\A} \Om^1$ is an algebraic analog of the map $\nabla : \Ga\left( E\right) \to \Ga\left( E \otimes T^*\left( M\right) \right)$ given by \eqref{comm_alg_conn}. 
\end{rem}

\begin{prop}\label{conn_prop}\cite{connes:ncg94}
	Following conditions hold:	
	\begin{enumerate} 
		\item[(a)] 	Let $e \in \End_{\A}\left( \E\right)$ be an idempotent and $\nabla$ is a connection on $\E$; then 
	\be\label{idem_conn}
		\xi \mapsto \left(e \otimes 1 \right) \nabla \xi
	\ee
		is a connection on $e\E$,
		\item[(b)] Any finite projective module $\E$ admits a connection,
		\item[(c)]  The space of connections is an affine space over the vector space $\Hom_{\sA}\left(\E, \E \otimes_{\A} \Om^1 \right)$, 
		\item[(d)] Any connection $\nabla$ extends uniquely up to a linear map of  $\widetilde{\mathcal E}= \mathcal E \otimes_{\A} \Om$ into itself
		such that
	\be
	\nabla\left(\xi \otimes \om \right) = \nabla\left(\xi \right) \om + \xi \otimes d\om; ~~\forall \xi \in \mathcal E, ~ \om \in \Om. 
	\ee
	\end{enumerate}
\end{prop}
A \textit{curvature} of a connection $\nabla$ is a (right $\A$-linear) map
\be
F_\nabla : \mathcal E \to \mathcal E \otimes_{\A} \Om^2
\ee
defined as a restriction of $\nabla \circ \nabla$ to $\mathcal E$, that is, $F_{\nabla} = \nabla \circ \nabla |_{\mathcal E}$. 
A connection is said to be \textit{flat} if
its curvature is identically equal to $0$ (cf. \cite{brzezinsky:flat_co}).
\begin{remark}
	Above algebraic notions of curvature and flat connection are generalizations of corresponding geometrical notions explained in \cite{koba_nomi:fgd} and the Section \ref{geom_flat_subsec}.
\end{remark}
For any projective $\A$ module $\mathcal E$ there is a \textit{trivial connection}

\bean
\nabla:  \mathcal E \otimes_{\A} \Om \to \mathcal E \otimes_{\A} \Om, \\
\nabla = \Id_{\mathcal E} \otimes d.
\eean
From $d^2 = d \circ d = 0$ it follows that $\left(\Id_{\mathcal E} \otimes d \right) \circ   \left(\Id_{\mathcal E} \otimes d \right)$ = 0, i.e. any trivial connection is flat.
\begin{lemma}\label{flat_rem}
	If $\nabla:  \mathcal E \to \mathcal E \otimes_{\A} \Om^1$ is a trivial connection and $e \in \End_{\A}\left( \E\right)$ is an idempotent then the given by \eqref{idem_conn}
	$$
\xi \mapsto \left(e \otimes 1 \right) \nabla \xi	
	$$
	connection $\nabla_e: e\mathcal E \to e\mathcal E \otimes \Om^1$  on $e\mathcal E$ is flat.
\end{lemma} 
\begin{proof}
	From
	$$
\left(e \otimes 1 \right)\left(\Id_{\mathcal E} \otimes d \right) \circ \left(e \otimes 1 \right)\left(\Id_{\mathcal E} \otimes d \right) = e \otimes d^2 = 0	
	$$
	it turns out that $\nabla_e \circ \nabla_e = 0$, i.e. $\nabla_e$ is flat.
\end{proof}
\begin{remark}
The notion of the trivial connection is an algebraic version of geometrical canonical connection explained in the Section \ref{geom_flat_subsec}.
\end{remark}

\subsubsection{Spectral triples}

\paragraph{}
This section contains citations of  \cite{hajac:toknotes}. 
\paragraph{Definition of spectral triples}
\begin{defn}
	\label{df:spec-triple}\cite{hajac:toknotes}
	A (unital) {\it {spectral triple}} $(\A, \H, D)$ consists of:
	\begin{itemize}
		\item
		a pre-$C^*$-algebra $\A$ with an involution $a \mapsto a^*$, equipped
		with a faithful representation on:
		\item
		a \emph{Hilbert space} $\H$; and also
		\item
		a \emph{selfadjoint operator} $D$ on $\mathcal{H}$, with dense domain
		$\Dom D \subset \H$, such that $a(\Dom D) \subseteq \Dom D$ for all 
		$a \in \mathcal{A}$.
	\end{itemize}
\end{defn}
There is a set of axioms for  spectral triples described in \cite{hajac:toknotes,varilly:noncom}.

\paragraph{Noncommutative differential forms}
\paragraph*{} 
Any spectral triple naturally defines a cycle $\rho : \A \to \Om_D$ (cf. Definition \ref{conn_defn}). 
In particular for any spectral triple there is an $\A$-module $\Om^1_D\subset B\left(\H \right) $ of order-one differential forms which is a linear span of operators given by
\begin{equation}\label{dirac_d_module}
a\left[D, b \right];~a,b \in \A.
\end{equation}
There is the differential map
\begin{equation}\label{diff_map}
\begin{split}
d: \A \to \Om^1_D, \\
a \mapsto \left[D, a \right].
\end{split}
\end{equation}

\section{Noncommutative finite-fold coverings}

\subsection{Coverings of $C^*$-algebras}
\begin{definition}
	If $A$ is a $C^*$- algebra then an action of a group $G$ is said to be {\it involutive } if $ga^* = \left(ga\right)^*$ for any $a \in A$ and $g\in G$. The action is said to be \textit{non-degenerated} if for any nontrivial $g \in G$ there is $a \in A$ such that $ga\neq a$. 
\end{definition}
\begin{definition}\label{fin_def_uni}
	Let $A \hookto \widetilde{A}$ be an injective *-homomorphism of unital $C^*$-algebras. Suppose that there is a non-degenerated involutive action $G \times \widetilde{A} \to \widetilde{A}$ of a finite group $G$, such that $A = \widetilde{A}^G\stackrel{\text{def}}{=}\left\{a\in \widetilde{A}~|~ a = g a;~ \forall g \in G\right\}$. There is an $A$-valued product on $\widetilde{A}$ given by
	\begin{equation}\label{finite_hilb_mod_prod_eqn}
	\left\langle a, b \right\rangle_{\widetilde{A}}=\sum_{g \in G} g\left( a^* b\right) 
	\end{equation}
	and $\widetilde{A}$ is an $A$-Hilbert module. We say that a triple $\left(A, \widetilde{A}, G \right)$ is an \textit{unital noncommutative finite-fold  covering} if $\widetilde{A}$ is a finitely generated projective $A$-Hilbert module.
\end{definition}
\begin{remark}
	Above definition is motivated by the Theorem \ref{pavlov_troisky_thm}.
\end{remark}
\begin{theorem}\label{pavlov_troisky_thm}\cite{pavlov_troisky:cov}. 
	Suppose $\mathcal X$ and $\mathcal Y$ are compact Hausdorff connected spaces and $p :\mathcal  Y \to \mathcal X$
	is a continuous surjection. If $C(\mathcal Y )$ is a projective finitely generated Hilbert module over
	$C(\mathcal X)$ with respect to the action
	\begin{equation*}
	(f\xi)(y) = f(y)\xi(p(y)), ~ f \in  C(\mathcal Y ), ~ \xi \in  C(\mathcal X),
	\end{equation*}
	then $p$ is a finite-fold  covering.
\end{theorem} 
\subsection{Coverings of spectral triples}\label{triple_fin_cov}

\begin{defn}\label{triple_lift_defn}Let  $\left( \A, \H, D\right)$ 
	be a spectral triple, and let $A$ be the $C^*$-norm completion of $\A$. Let $\left(A, \widetilde{A}, G \right)$ be an unital noncommutative finite-fold covering such that there is the dense inclusion $\A \hookto A$. Let $\widetilde{\H} \stackrel{\text{def}}{=} \widetilde{A} \otimes_A \H$ is a Hilbert space such that the Hilbert product $\left(\cdot, \cdot \right)_{\widetilde{\H}}$ is given by
	$$
	\left(a \otimes \xi, b \otimes \eta \right)_{\widetilde{\H}} = \frac{1}{\left|G\right|} \left(\xi, \left( \sum_{g \in G} g\left(\widetilde{a}^*\widetilde{b} \right)  \right) \eta \right)_{\H}; ~\forall \widetilde{a},\widetilde{b} \in \widetilde{A},~ \xi, \eta \in \H 
	$$
	where $\left(\cdot, \cdot \right)_{{\H}}$ is the Hilbert product on $\H$.
	There is the natural representation $\widetilde{A} \to B\left(\widetilde{\H} \right)$. A spectral triple $\left( \widetilde{\A}, \widetilde{\H}, \widetilde{D}\right)$ is said to be a $\left(A, \widetilde{A}, G \right)$-\textit{lift} of $\left( \A, \H, D \right)$ if following conditions hold:
	\begin{enumerate}
		\item[(a)] $\widetilde{A}$ is a $C^*$-norm completion of $\widetilde{\A}$,
		\item [(b)] 
		$\widetilde{D} \left(1_{\widetilde{A}} \otimes_{A} \xi\right)= 1_{\widetilde{A}} \otimes_{A} D \xi;~ \forall \xi \in \Dom D,$
		\item[(c)]	$
		\widetilde{D}\left(  g \widetilde{\xi}\right) = 	g\left( \widetilde{D} \widetilde{\xi}\right)$ for any $\widetilde{\xi}\in \Dom\widetilde{D}, ~  g \in G$.
	\end{enumerate}
\end{defn} 
\begin{remark}
	It is proven in \cite{ivankov:qncstr} that for any spectral triple  $\left( \A, \H, D\right)$ 
	and any unital noncommutative finite-fold  covering $\left(A, \widetilde{A}, G \right)$ there is the unique $\left(A, \widetilde{A}, G \right)$-lift $\left(\widetilde{ \A}, \widetilde{\H}, \widetilde{D} \right)$ of $\left( \A, \H, D \right)$.
\end{remark}
\begin{remark}
	It is known that if $M$ is a Riemannian manifold and $\widetilde{M} \to M$ is a covering, then $\widetilde{M}$ has the natural structure of Riemannian manifold (cf. \cite{koba_nomi:fgd}).
	The existence of  lifts of  spectral triples is a noncommutative generalization of this fact (cf. \cite{ivankov:qncstr})
\end{remark}

\section{Construction of noncommutative flat coverings}\label{nc_flat_sec}
\paragraph*{}
Let $\left( \A, \H, D\right)$ be a spectral triple, let  $\left( \widetilde{\A}, \widetilde{\H}, \widetilde{D}\right)$ is the $\left(A, \widetilde{A}, G \right)$-lift of $\left( \A, \H, D \right)$. Let $V = \C^n$ and with left action of $G$, i.e. there is a linear representation $\rho: G \to GL\left(\C,n\right)$. Let $\widetilde{\mathcal E} =  \A \otimes \C^{n} \approx \widetilde{\A}^n$ be a free module over $\widetilde{\A}$, so $\widetilde{\mathcal E}$ is a projective finitely generated $\A$-module (because $\widetilde{ \A}$ is a finitely generated projective $\A$-module). Let $\widetilde{\nabla} : \widetilde{\mathcal E} \to \widetilde{\mathcal E} \otimes_{\widetilde{\A}} \Om^1_{\widetilde{D}}$ be the trivial flat connection. In \cite{ivankov:qncstr} it is proven that $\Om^1_{\widetilde{D}} = \widetilde{\A}\otimes_{\A}\Om^1_{D}$ it follows that the connection $\widetilde{\nabla} : \widetilde{\mathcal E} \to \widetilde{\mathcal E} \otimes_{\widetilde{\A}} \Om^1_{\widetilde{D}}$ can be regarded as a map $\nabla':\widetilde{\mathcal E} \to \widetilde{\mathcal E} \otimes_{\widetilde{\A}} \widetilde{\A} \otimes_{\A} \Om^1_{\widetilde{D}}= \widetilde{\mathcal E} \otimes_{\A}\Om^1_{{D}}$, i.e. one has a connection
\be\nonumber
\nabla':\widetilde{\mathcal E} \to  \widetilde{\mathcal E} \otimes_{\A}\Om^1_{{D}}.
\ee
From $\widetilde{\nabla} \circ \widetilde{\nabla} |_{\mathcal E}=0$ it turns out that  $\nabla' \circ \nabla' |_{\mathcal E}=0$, i.e. $\nabla'$ is flat. There is the action of $G$ on $\widetilde{\mathcal E}= \widetilde{\A} \otimes \C^n$ given by 
\be
g\left( \widetilde{a}\otimes  x\right)   = g \widetilde{a} \otimes g x; ~~ \forall g \in G,~ \widetilde{a} \in \widetilde{\A}, ~ x \in \C^n.
\ee
Denote by
\be
\mathcal E = \widetilde{\mathcal E}^G = \left\{\widetilde{\xi} \in  \widetilde{\mathcal E}~|~ G\widetilde{\xi} = \widetilde{\xi}\right\}
\ee
Clearly $\mathcal E$ is an $\A$-$\A$-bimodule. For any $\widetilde{\xi} \in \widetilde{\mathcal E}$ there is the unique decomposition
\be
\begin{split}
	\widetilde{\xi} = \xi + \xi_\perp, \\
	\xi = \frac{1}{\left|G\right|}\sum_{g \in G} g \widetilde{\xi},\\
\xi_\perp = \widetilde{\xi} - \xi. 
\end{split}
\ee
From the above decomposition it turns out the direct sum $\widetilde{\mathcal E} = \widetilde{\mathcal E}^G \bigoplus {\mathcal E}_\perp$ of $\A$-modules. So  $\mathcal E = \widetilde{\mathcal E}^G$ is a projective finitely generated $\A$-module, it follows that there is an idempotent $e \in \End_{\A}{\widetilde{\mathcal E}}$ such that $\mathcal E = e \widetilde{\mathcal E}$. The Proposition \ref{conn_prop} gives the canonical connection 
\be\label{nc_flat_conn}
\nabla : \mathcal E \to \mathcal E \otimes_{\A} \Om^1_D
\ee
which is defined by the connection $\nabla':\widetilde{\mathcal E} \to  \widetilde{\mathcal E} \otimes_{\A}\Om^1_{{D}}$ and the idempotent $e$. From the Lemma \ref{flat_rem} it turns out that $\nabla$ is flat.
\begin{definition}
We say that  $\nabla$ is a \textit{flat connection induced by} noncommutative covering $\left(A, \widetilde{A}, G\right)$ and the	linear representation $\rho: G \to GL\left(\C,n\right)$, or we say the $\nabla$ \textit{comes from the representation} $\rho: G \to GL\left(\C,n\right)$.
\end{definition}
\break
\section{Mapping between geometric and algebraic constructions}

\paragraph{}
The geometric (resp. algebraic) construction of flat connection is explained in the Section \ref{geom_flat_subsec} (resp. \ref{nc_flat_sec}). Following table gives a mapping between these constructions.
\newline

\begin{tabular}{|c|c|c|}
	\hline
	&Geometry & Agebra\\
	\hline
&	&\\
1&Riemannian manifold $M$.  & Spectral triple $\left(\Coo\left(M \right), L^2\left(M, \sS \right), \slashed D   \right)$. \\ & & \\
2&Topological covering $\widetilde{M} \to M$. & Noncommutative covering, \\ & & $\left(C\left(M \right), C\left(\widetilde{M} \right), G\left(\widetilde{M}~|~M \right)   \right),$  \\
& & given by the Theorem \ref{pavlov_troisky_thm}. \\ & & \\
3&Natural structure of Reimannian  & Triple $\left(\Coo\left(\widetilde{M} \right), L^2\left(\widetilde{M}, \widetilde{\sS} \right), \widetilde{\slashed D}   \right)$ is the\\
& manifold on the covering space $\widetilde{M}$.&  $\left(C\left(M \right), C\left(\widetilde{M} \right), G\left(\widetilde{M}~|~M \right)   \right)$ -lift
  \\
&  & of $\left(\Coo\left(M \right), L^2\left(M, \sS \right), \slashed D   \right)$.\\ & & \\
4&Group homomorphism   & Action $ G\left(\widetilde{M}~|~M \right) \times \C^n \to \C^n$\\
&$ G\left(\widetilde{M}~|~M \right) \to GL\left(n, C \right)$ & \\ & & \\
5&Trivial bundle $\widetilde{M}\times \C^n$. & Free module $\Coo\left(\widetilde{M} \right) \otimes \C^n$. \\ & & \\
6&Canonical flat connection on $\widetilde{M}\times \C^n$ & Trivial flat connection on $\Coo\left(\widetilde{M} \right) \otimes \C^n$\\ & & \\
7&Action of $G\left(\widetilde{M}~|~M \right)$ on $\widetilde{M}\times \C^n$  & Action of $G\left(\widetilde{M}~|~M \right)$ on  $\Coo\left(\widetilde{M} \right) \otimes \C^n$\\ & & \\
8&Quotient space  & Invariant module    \\
& $P = \left( \widetilde{M}\times \C^n\right)/G\left(\widetilde{M}~|~M \right).$& $\mathcal E =  \left( \Coo\left( \widetilde{M}\right) \otimes \C^n\right)^{G\left(\widetilde{M}~|~M \right)}$ \\ & & \\
9&Geometric flat connection on $P$ & Algebraic flat connection on $\mathcal E$.\\ & & \\
	\hline
\end{tabular}

\section{Noncommutative examples}

\subsection{Noncommutative tori}
\paragraph*{}
Following text is the citation of \cite{ivankov:qncstr}.
If $\Th$ be  a real skew-symmetric $n \times n$ matrix. There is a $C^*$- algebra $C\left(\T^n_\Th \right)$ which is said to be the \textit{noncommutative torus} (cf. \cite{ivankov:qncstr}). There is a  pre-$C^*$-algebra $\Coo\left(\T^n_\Th \right)$ and the spectral triple $\left(\Coo\left(\T^n_\Th \right), \H, D \right)$ such that it is the dense inclusion $\Coo\left(\T^n_\Th \right) \hookto C\left(\T^n_\Th \right)$.  If  $\overline{k} = \left(k_1, ..., k_n\right) \in \mathbb{N}^n$ and
$$
\widetilde{\Theta} = \begin{pmatrix}
0& \widetilde{\theta}_{12} &\ldots & \widetilde{\theta}_{1n}\\
\widetilde{\theta}_{21}& 0 &\ldots & \widetilde{\theta}_{2n}\\
\vdots& \vdots &\ddots & \vdots\\
\widetilde{\theta}_{n1}& \widetilde{\theta}_{n2} &\ldots & 0
\end{pmatrix}
$$
is a skew-symmetric matrix such that
\begin{equation*}
e^{-2\pi i \theta_{rs}}= e^{-2\pi i \widetilde{\theta}_{rs}k_rk_s}
\end{equation*}
then one has a following theorem.
\begin{thm}\label{nt_fin_cov_thm}\cite{ivankov:qnc}												
	The triple $\left(C\left(\mathbb{T}^n_{\Th}\right), C\left(\mathbb{T}^n_{\widetilde{\Th}}\right),\mathbb{Z}_{k_1}\times...\times\mathbb{Z}_{k_n}\right)$  is an unital noncommutative finite-fold  covering.
\end{thm}
There is $\left(C\left(\mathbb{T}^n_{\Th}\right), C\left(\mathbb{T}^n_{\widetilde{\Th}}\right),\mathbb{Z}_{k_1}\times...\times\mathbb{Z}_{k_n}\right)$-lift $\left(\Coo\left(\T^n_{\widetilde{\Th}} \right), \widetilde{\H}, \widetilde{D} \right)$ of $\left(\Coo\left(\T^n_\Th \right), \H, D \right)$. From the construction of the Section \ref{nc_flat_sec} it follows that for any representation $\rho:\mathbb{Z}_{k_1}\times...\times\mathbb{Z}_{k_n} \to GL\left(N, \C \right)$ there is a finitely generated $\Coo\left(\T^n_\Th \right)$-module $\mathcal E$ and a flat connection 
$$
\mathcal E \to \mathcal E \otimes_{\Coo\left(\T^n_\Th \right)}\otimes \Om^1_D
$$ 
which comes from $\rho$.
\subsection{Isospectral deformations}
\paragraph*{}
A very general construction of isospectral
deformations
of noncommutative geometries is described in \cite{connes_landi:isospectral}. The construction
implies in particular that any
compact Spin-manifold $M$ whose isometry group has rank
$\geq 2$ admits a
natural one-parameter isospectral deformation to noncommutative geometries
$M_\theta$.
We let $\left(\Coo\left(M \right)  ,  L^2\left(M,S \right)  , \slashed D\right)$ be the canonical spectral triple associated with a
compact spin-manifold $M$. We recall that $C^\infty(M)$ is
the algebra of smooth
functions on $M$, $S$ is the spinor bundle and $\slashed D$
is the Dirac operator.
Let us assume that the group $\mathrm{Isom}(M)$ of isometries of $M$ has rank
$r\geq2$.
Then, we have an inclusion
\begin{equation*}
\mathbb{T}^2 \subset \mathrm{Isom}(M) \, ,
\end{equation*}
with $\mathbb{T}^2 = \mathbb{R}^2 / 2 \pi \mathbb{Z}^2$ the usual torus, and we let $U(s) , s \in
\mathbb{T}^2$, be
the corresponding unitary operators in $\H = L^2(M,S)$ so that by construction
\begin{equation*}
U(s) \, \slashed D = \slashed D \, U(s).
\end{equation*}
Also,
\begin{equation}\label{isospectral_sym_eqn}
U(s) \, a \, U(s)^{-1} = \alpha_s(a) \, , \, \, \, \forall \, a \in \mathcal{A} \, ,
\end{equation}
where $\alpha_s \in \mathrm{Aut}(\mathcal{A})$ is the action by isometries on the
algebra of functions on
$M$. In \cite{connes_landi:isospectral} is constructed a spectral triple $\left(l\Coo\left(M \right)  ,  L^2\left(M,S \right)  , \slashed D\right)$ such that $l\Coo\left(M \right)$ is a noncommutative algebra which is said to be an \textit{isospectral deformation} of $\Coo\left(M\right)$. For any finite-fold topological covering $\widetilde{M}\to M$ there is the  finite-fold noncommutative covering $\left(lC\left( \widetilde{M}\right), l\left(M \right), G\left( \widetilde{M}~|M\right)    \right)$ (cf. \cite{ivankov:qnc}). So there is the   $\left(lC\left( \widetilde{M}\right), l\left(M \right), G\left( \widetilde{M}~|M\right)    \right)$-lift $$\left(l\Coo\left(\widetilde{M} \right)  ,  L^2\left(\widetilde{M},\widetilde{S} \right)  , \widetilde{\slashed D}\right)$$ of $\left(l\Coo\left(M \right)  ,  L^2\left(M,S \right)  , \slashed D\right)$. From the construction of the Section \ref{nc_flat_sec} it follows that for any representation $\rho:G\left( \widetilde{M}~|M\right) \to GL\left(N, \C \right)$ there is a finitely generated $l\Coo\left(M \right) $-module $\mathcal E$ and a flat connection
$$
\mathcal E \to \mathcal E \otimes_{l\Coo\left(M \right)}\otimes \Om^1_D
$$ 
which comes from $\rho$.

\end{document}